\documentclass[10pt]{article}

\usepackage{color, amsmath,amssymb, amsfonts, amstext,amsthm, latexsym}
\usepackage{epsfig}
\usepackage{longtable}
\usepackage{graphicx}
\usepackage{subfigure}

\newtheorem{theorem}{Theorem}
\newtheorem{lemma}{Lemma}
\newtheorem{definition}{Definition}

  \newtheorem{Main Results}[theorem]{MainResults}
\newtheorem{proposition}[theorem]{Proposition}
 \newtheorem{remark}{Remark}
  \newtheorem*{assumption}{Assumption}

\begin{document}

\title{Derivation of Fokker-Planck equations for stochastic dynamical systems  under excitation of multiplicative non-Gaussian white noise}
 \author{\small{\it{Xu Sun$^{1,2,}$, Jinqiao Duan$^2$,  Xiaofan Li$^2$,  Hua Liu$^1$, Xiangjun Wang$^1$, and Yayun Zheng$^1$}}\\
 \\\small{$^1$Huazhong University of Science and Technology,
  Wuhan 430074, Hubei, China}
   \\\small{$^2$Illinois Institute of Technology,
  Chicago, IL 60616, USA}\\
  }

\date{March 1, 2016}
\maketitle

\begin{abstract}
Fokker-Planck equations describe  time evolution of probability densities of stochastic dynamical systems and play  an important role in quantifying propagation and evolution of uncertainty. Although Fokker-Planck equations can be written explicitly for nonlinear dynamical systems excited by Gaussian white noise, they are not available in general for nonlinear dynamical systems excited by  multiplicative non-Gaussian white noise. Marcus stochastic differential equations are often appropriate models in engineering and physics for stochastic dynamical systems excited by non-Gaussian white noise. In this paper, we derive explicit forms of Fokker-Planck equations for  one dimensional systems modeled by Marcus stochastic differential equations under multiplicative non-Gaussian white noise. As examples to illustrate the theoretical results, the derived formula is used to obtain Fokker-Plank equations for nonlinear dynamical systems under excitation of (i) $\alpha$-stable white noise; (ii) combined Gaussian and Poisson white noise, respectively.


\end{abstract}

\maketitle

\pagestyle{plain}

\section{Introduction and statement of the problem}  \label{intro}

Stochastic differential equations(SDEs) are often appropriate models for dynamical systems subjected to random excitations \cite{Oksendal2003}.  Fokker-Planck equation describes the evolution of probability density functions and is an important tool to study how uncertainties propagate and evolve in dynamical systems \cite{Oksendal2003,Klebaner2005}. Nonlinear dynamical systems excited by Gaussian white noise are often modeled by SDEs driven by Brownian motions (or Wiener processes). For    SDEs driven by Brownian motions, there are explicit formulas to obtain the associated Fokker-Planck equations, regardless the SDEs are in sense of It\^{o} or Stratonovich \cite{Oksendal2003,Klebaner2005}.

Nonlinear dynamical systems excited by non-Gaussian white noise are often modeled by SDEs driven by non-Gaussian L\'{e}vy processes. For SDEs driven by  non-Gaussian L\'evy processes, there are two popular definitions, i.e., SDEs defined in sense of It\^{o} and those defined in sense of Marcus  \cite{Marcus1978, Marcus1981, KurtzPardouxProtter1995, Applebaum2009}.
Marcus SDEs are  often appropriate models in engineering and scientific practice \cite{Marcus1978, Marcus1981, KurtzPardouxProtter1995, Applebaum2009}. It is recently shown \cite{SunDuanLi2013,KanazawaSagawaHayakawa2012,SunDuanLi2016} that  Marcus SDE is equivalent to the well known DiPaola-Falsone SDE \cite{DiPaolaFalsone1993,DiPaolaFalsone1993b} which is widely used in engineering and physics. Comparison of Marcus integral and Stratonovich integral is recently discussed in \cite{ChechkinPavlyukevich2014} for systems with jump noise.

Solutions of both It\^{o} and Marcus SDEs are Markov processes \cite{Applebaum2009}. It is well known that Fokker-Planck equations for Markov processes contain  the adjoint operator of the infitesimal generator of the Markov processes\cite{Applebaum2009}.  Unlike the Gaussian cases, Fokker-Planck equations  for SDEs driven by  non-Gaussian L\'{e}vy processes are not readily available due to the difficulty in obtaining the explicit expressions for the adjoint  of the infinitesimal generators  associated with these SDEs \cite{Applebaum2009}.  While Fokker-Planck equations for It\^{o} SDEs have been discussed by many authors, see \cite{SunDuan2012,Schertzer2001} among others, the research on Fokker-Planck equations for Marcus SDEs is relatively rare. The only published result we can find so far about Fokker-Planck equations for Marcus SDEs is presented in \cite{SunDuan2012}, where an explicit form of Fokker-Planck equations is derived for Marcus SDEs under the condition that coefficients of the noise terms are strictly nonzero. However, it is still not clear what the Fokker-Planck equations are like for Marcus stochastic differential equations under general conditions.

L\'{e}vy processes are  stochastic processes with properties of independent  and stationary increments, as well as  stochastically continuous sample paths \cite{Applebaum2009, Sato1999}. Examples of L\'{e}vy processes include Brownian motions, compound Poisson processes, $\alpha$-stable processes and so on.   A one-dimensional  L\'{e}vy process $L(t)$, taking values in $\mathbb{R}$, is
characterized by a drift   $b\in\mathbb{R}$ , a
 positive real number $A$ and a    measure $\nu$ defined on
$\mathbb{R}$ and concentrated on $\mathbb{R}\backslash\{0\}$. In fact, this measure $\nu$
  satisfies the following condition \cite{Applebaum2009}
\begin{align}\label{s1_1}
\int_{\mathbb{R}\backslash \{0\} }(y^2\wedge 1)\nu({\rm d}y)<\infty,
\end{align}
where $y^2 \wedge 1$  represents the minimum of $y^2$ and $1$. This measure $\nu$ is  called a L\'{e}vy jump measure for  the
L\'{e}vy process $L(t)$. A L\'{e}vy process with the generating
triplet $(b, A, \nu)$ has the L\'{e}vy-It\^{o} decomposition
\begin{align}\label{s1_4}
{\rm d} L(t) =b{\rm d} t+{\rm d} B(t)+ \int_{ |y |< 1}y\tilde{N}({\rm d} t, {\rm d} y)+\int_{ |y |\geq 1}y{N}({\rm d} t,
{\rm d} y).
\end{align}
where $N({\rm d}t, {\rm d}x)$ is the Poisson random measure, $\tilde{N}({\rm d}t,
{\rm d}x)=N({\rm d}t, {\rm d}x)-\nu({\rm  d}x){\rm d}t$  is the compensated Poisson random
measure, and $B(t)$ is  the Brownian motion (i.e., Wiener process) with
variance $A$.

Different kinds of L\'{e}vy processes can be obtained by taking different values of the triplet $(b, A, \nu)$. Just as a Gaussian white noise can be regarded as the formal derivative of a Brownian motion process, a non-Gaussian white noise can be regarded as  the formal derivative of some non-Gaussian L\'evy process.

We shall consider stochastic dynamical systems described by the following  SDE
 in sense of Marcus,
\begin{equation}\label{s1_5}
{\rm d}X(t)=f(X(t)){\rm d}t+\sigma(X(t) )\diamond {\rm d}L(t),
\end{equation}
where $X(t)$ is a scalar process, and $L(t)$ is the one-dimensional L\'{e}vy process with the generating triplet     $(b, A, \nu)$.
Note that we only consider one-dimensional case in this paper.
The solution of equation  (\ref{s1_5})    is interpreted as
\begin{align}\label{s1_6}
X(t) =X(t) + \int_0^t f(X(s) ){\rm d}s+ \int_0^t \sigma(X(s-))\diamond {\rm d} L(s),
 \end{align}
where $X(s-) =\lim\limits_{u<s, u\to s} X(u)$, and "$\diamond$" indicates Marcus integral \cite{Marcus1978, Marcus1981, Applebaum2009}   defined by
\small
\begin{align}\label{s1_7}
 \int_0^t \sigma(X_{s-})\diamond {\rm d} L(s)&=
\int_0^t \sigma(X_{s-})  {\rm d}L(s) + \frac{A}{2} \int_0^t \sigma(X(s-) ) \sigma^\prime (X(s-)) {\rm d} s\nonumber\\
 &\quad + \sum_{0\leq s \leq t} \left[ \xi(\Delta L(s),  X(s-)) -X(s-) -\sigma(X(s-))\Delta L(s) \right],
 \end{align}
 \normalsize
 with $ \xi(r,   x)$   being the value at $z=1$ of the solution of the following ordinary differential equation(ODE):
\begin{align}\label{s1_8}
\frac{\rm d}{{\rm d} z} y(z) =r\sigma(y(z)), \quad\quad y(0)=x.
\end{align}
The first term at the right hand side of (\ref{s1_7}) is the It\^{o} integral, and the second term is the correction term due to the continuous component of $L(t)$, as also appears in Stratonovich integral, and the last term is the correction term due to jumps.  Note that the correction term due to  jumps  is expressed as the sum of some recursively infinite series in DiPaola-Falsone SDEs \cite{DiPaolaFalsone1993,DiPaolaFalsone1993b}, which have been extensively used in engineering and physics. Since it has been shown that Marcus SDEs and DiPaola-Falsone SDEs are essentially equivalent \cite{SunDuanLi2013}, the result for Marcus SDEs in this paper is also true for DiPaola-Falsone SDEs. For more discussion on the relationship between Marcus and DiPaola-Falsone SDEs, readers are referred to \cite{SunDuanLi2013,SunDuanLi2016}.

In this paper, we shall derive explicit forms of Fokker-Planck equations which govern the probability density functions for the solution of the SDE (\ref{s1_5}).  The result presented here is applicable under much more general conditions than that in \cite{SunDuan2012}. While the coefficient  of the noise term is required to be strictly nonzero in \cite{SunDuan2012}, the result here allows the coefficient $\sigma$ to have finite or countable zeros.

Let $p(x,t\big | X(0)=x_0)$ represent the probability density function for the solution $X(t)$ of the SDE (\ref{s1_5}), and for convenience, we drop off the initial condition and simply use $p(x,t)$ instead of $p(x,t\big | X(0)=x_0)$. Throughout this paper, we assume the following.

\begin{assumption}[H1]
The probability density function $p(x,t)$ for the solution $X(t)$ of (\ref{s1_5}) exists, and $p(x,t)$ is continuously differentiable with respect to $t$ and twice continuously differentiable with respect to $x$.
\end{assumption}

We are not going to present conditions for existence and regularity of the solution and the probability density associated with the SDE (\ref{s1_5}), which is out of the scope of this paper.  Note that the existence and regularity of probability density for solutions of SDEs driven by L\'evy processes are  active research topics itself.

The sections of the paper are organized as follows. In subsection 2.1, we  derive Fokker-Planck equations for  systems modeled by (\ref{s1_5}) with $\sigma(x)\ne 0$. The  condition that $\sigma(x) \ne 0$ is relaxed in section subsection 2.2 by assuming $\sigma(x)$ has finite zeros. In section 3, we apply the theoretical result presented in section 2 to obtain the Fokker-Planck equations for some nonlinear dynamical systems under excitation of $\alpha$-stable white noise, or combined Gaussian and Poisson white noise.  Section 4 is the conclusion.

\section{Derivation of Fokker-Planck equation}
\label{sec2}

We derive Fokker-Planck equations for the SDE (\ref{s1_5}) in two cases: (i) $\sigma$ has no zero, and (ii) $\sigma$ has zeros.

\subsection{Cases where $\sigma$ has no zero}
In this subsection, we derive Fokker-Planck equations  for  the SDE (\ref{s1_5}) with a different approach from that in  \cite{SunDuan2012}. The advantage of the approach here lies in that it can be modified to be applicable in cases where $\sigma(x)$ has zeros.

\begin{definition}
The transform $H$ associated with the coefficient $\sigma$ in the SDE (\ref{s1_5})   is defined as
\begin{align}\label{s1_10}
H(x) = \int_a^x \frac{{\rm d}t}{\sigma(t)},
\end{align}
where $a$ is any   constant.
\end{definition}

\begin{lemma}\label{lemma1}
Suppose $\sigma$ is Lipschitz and has no zero, then the transform $H$ has the following properties:\\
(i) $H$ is well defined and monotone on $(-\infty,\infty)$;\\
(ii) $H$ is bijective (i.e., one-to-one and onto) and maps from $(-\infty,\infty)$ to $(-\infty,\infty)$;\\
(iii) $H$ has the inverse transform $H^{-1}$, and $H^{-1}$ is bijective and maps from $(-\infty,\infty)$ to $(-\infty,\infty)$;\\
(iv) $\forall y\in \mathbb{R}$, $H^{-1}(H(\cdot)+y)$  is bijective  and maps from $(-\infty,\infty)$ to $(-\infty,\infty)$.
\end{lemma}

\begin{proof}[Proof of Lemma 1]
Conclusion (i) follow from the fact that $\sigma$ is Lipschitz continuous and has no zero.

To show (ii), since $H$ is monotone and defined on $(-\infty,\infty)$, it is sufficient to show that $H$ goes to infinity as $x$ goes to infinity. Without loss of generality, let us assume $\sigma>0$. It follows from  (\ref{s1_10}) that for $x>a$,
\begin{align}\label{a1_1}
 H(x)  & = \int_a^x \frac{{\rm d} t}{|\sigma(t)|} = \int_a^x \frac{{\rm d} t}{|\sigma(t)-\sigma(a) + \sigma(a)|} \ge  \int_a^x \frac{{\rm d} t}{|\sigma(t)-\sigma(a)| + |\sigma(a)|}\nonumber\\
& \ge \int_a^x \frac{{\rm d} t}{L|t-a| + |\sigma(a)|},
\end{align}
where $L$ is the Lipshitz constant satisfying $|\sigma(t)-\sigma(a)| \le L|t-a|$.
It follows from (\ref{a1_1}) that $H(x) \to +\infty$ as $x\to +\infty$. Similarly, $H(x)\to -\infty$ as $x\to -\infty$ for $\sigma>0$.

  (ii) implies (iii) and (iv).
\end{proof}

Now for the case $\sigma$ is nonzero, we present the Fokker-Planck equation for the SDE (\ref{s1_5}) in the following proposition.

\begin{proposition}\label{proposition1}
Suppose the assumption H1 holds,  $f$ is differentiable, $\sigma$ is Lipschitz continuous and twice differentiable, and $\sigma(x)\ne 0$ for all $x\in \mathbb{R}$, then the probability density function $p(x,t)$ for the solution $X(t)$ of the SDE (\ref{s1_5}) satisfies the following Fokker-Planck equation
 \begin{align}\label{s1_34}
  \frac{\partial p(x,t)}{\partial t}& =-  \frac{\partial}{\partial x} \left[\left(f(x)  + b\sigma (x)  + \frac{A}{2}   \sigma  (x) \sigma'(x)  \right)p(x,t)\right] + \frac{A}{2}   \frac{\partial^2}{\partial x^2}\left(\sigma ^2(x)p(x,t)\right)\nonumber\\
    &+ \int_{{\mathbf R}\backslash \{0\}}     \left[ \frac{\sigma(H^{-1} (H(x)- y)) }{\sigma(x)} p(H^{-1} (H(x)- y),t) \right.\nonumber\\
 &\quad\quad\quad\quad\quad\quad\quad\quad\quad \left. - p(x,t) +   y \,{\mathbf I}_{(-1,\;1)} (y)  \frac{\partial}{\partial x}(\sigma(x) p(x,t)) \right]\nu({\rm d}y),
\end{align}
where $H$ is the transform defined as in (\ref{s1_10}), and $H^{-1}$ is the inverse of $H$.
\end{proposition}

\begin{remark}
The Fokker-Planck equation (\ref{s1_34}) is the same as the one presented in \cite{SunDuan2012}. However, it is derived here in a different approach from that in \cite{SunDuan2012}.  The advantage of the approach here lies in that it can be modified to be applicable in cases where the coefficient $\sigma$ of the multiplicative noise in the SDE (\ref{s1_5}) has zeros.
\end{remark}

To prove Proposition \ref{proposition1}, we need the following Lemma, which can be found in many books on theory of distribution (e.g.  \cite{Bhattacharyya2012}).
\begin{lemma}\label{lemma2}
Suppose $\gamma_1 \in C (\mathbb{R})$ and $\gamma_2 \in C(\mathbb{R})$, if $\int_{\mathbb{R}}\phi(x) \gamma_1(x)\,{\rm d}x = \int_{\mathbb{R}}\phi(x) \gamma_2(x)\,{\rm d}x$ for all $\phi\in C_0^\infty(\mathbb{R})$, then $\gamma_1(x) = \gamma_2(x)$ for all $x \in \mathbb{R}$.
\end{lemma}

\begin{proof}[Proof of Proposition I]
It follows from (\ref{s1_4}), (\ref{s1_5}) and (\ref{s1_7}) that \cite{Applebaum2009}
\begin{align}\label{s1_9}
{\rm d}X(t) &=f(X(t)) {\rm d} t+b\sigma(X(t)) {\rm d}t + \sigma (X(t)) {\rm d} B(t) + \frac{A}{2}\sigma(X(t))\sigma'(X(t)) {\rm d}t \nonumber\\
&\quad + \int_{|y|< 1} [ \xi(y,   X(t-)) -X(t-)]\,\tilde N({\rm d} t,{\rm d}y )\nonumber\\
&\quad + \int_{|y| \ge 1}  [\xi(y,  X(t-)) -X(t-)]\,N({\rm d}t,{\rm d}y )\nonumber\\
&\quad + \int_{|y|< 1} [ \xi(y,   X(t-)) -X(t-) - \sigma(X(t-)\,y\, ]\,   \nu({\rm d}y ){\rm d}t.
\end{align}

It follows from the definition for $\xi$ from the ODE  (\ref{s1_8}) and the definition of the transform  $H$ in (\ref{s1_10}) that
\begin{align}\label{s1_11}
H( \xi(\Delta L(t),  X(t-)))- H(X(t-)) = \Delta L(t).
\end{align}
It follows from Lemma  \ref{lemma1} that
\begin{align}\label{s1_12}
\xi(\Delta L(t),   X(t-)) = H^{-1} (H(X(t-))+ \Delta L(t)).
\end{align}
Substituting (\ref{s1_12}) into (\ref{s1_9}), we get
\begin{align} \label{s1_13}
dX(t) &=f(X(t)) {\rm d} t+b\sigma(X(t)) {\rm d}t + \sigma (X(t)) {\rm d} B(t) + \frac{A}{2}\sigma(X(t))\sigma'(X(t)) {\rm d}t  \nonumber\\
&\quad + \int_{|y|< 1} [ H^{-1} (H(X(t-))+y) - X(t-)]\,\tilde N({\rm d}t,{\rm d}y )\nonumber\\
&\quad + \int_{|y| \ge 1}  [ H^{-1} (H(X(t-))+ y) - X(t-)]\,N({\rm d}t,{\rm d}y )\nonumber\\
&\quad + \int_{|y|< 1} [  H^{-1} (H(X(t-))+ y) -  X(t-) - \sigma(X(t-)\,y\, ]\,   \nu({\rm d}y ){\rm d}t.
\end{align}
By It\^o's formula \cite{Applebaum2009}, for  $\phi (x) \in C_0^\infty ({\mathbb R})$, it follows from (\ref{s1_13}) that
\small
\begin{align}\label{s1_14}
& \phi (X(t+\Delta t)) -\phi (X(t))  = \int_t^{t+\Delta t} \phi' (X(s-)) f(X(s-))\,{\rm d}s  +  \int_t^{t+\Delta t} b \phi' (X(s-))\sigma (X(s)) \,{\rm d}s\nonumber\\
&\quad +  \int_t^{t+\Delta t}  \phi' (X(s-) )\sigma (X(s))  \, {\rm d} B(s) + \frac{A}{2} \int_t^{t+\Delta t}  \phi'' (X(s-))\sigma ^2(X(s)) \,{\rm d}s \nonumber\\
&\quad  +\frac{A}{2} \int_t^{t+\Delta t}  \phi' (X(s-))\sigma  (X(s-)) \sigma'(X(s-)) \,{\rm d}s\nonumber\\
&\quad +\int_t^{t+\Delta t}\int_{|y|\ge 1} \left[ \phi\left( H^{-1} (H(X(s-))+ y) \right)-\phi(X(s-))\right]\,N({\rm d}s, {\rm d} y)\nonumber\\
&\quad + \int_t^{t+\Delta t}\int_{|y|< 1} \left[ \phi\left(H^{-1} (H(X(s-))+ y)\right)-\phi(X(s-))\right]\,\widetilde N({\rm d}s, {\rm d} y)\nonumber\\
&\quad + \int_t^{t+\Delta t}\int_{|y|< 1} \left[ \phi\left(H^{-1} (H(X(s-))+ y)\right)-\phi(X(s-))-  \phi'(X(s-))\,\sigma(X(s-))\,y \right]\,\nu ( {\rm d} y){\rm d}s\,.
\end{align}
\normalsize
Taking expectation at both sides of (\ref{s1_14}), we obtain
\footnotesize
\begin{align}\label{s1_15}
& \int_{-\infty}^{\infty} \phi (x)p(x, t+\Delta t) {\rm d} x  - \int_{-\infty}^{\infty} \phi (x) p(x,t) {\rm d}x \nonumber\\
 &= \int_t^{t+\Delta t} \int_{-\infty}^{\infty} \phi' (x) f(x)p(x,s)\,{\rm d}x\,{\rm d}s  + b \int_t^{t+\Delta t}\int_{-\infty}^{\infty}  \phi' (x) \sigma (x) p(x,s)\,{\rm d} x \,{\rm d}s\nonumber\\
 &\quad  +\frac{A}{2} \int_t^{t+\Delta t} \int_{-\infty}^{\infty} \phi' (x) \sigma  (x) \sigma'(x) p(x,s)\,{\rm d}x\,{\rm d}s +
   \frac{A}{2} \int_t^{t+\Delta t} \int_{-\infty}^{\infty} \phi'' (x)\sigma ^2(x)p(x,s)\,{\rm d}x\,{\rm d}s \nonumber\\
&\quad +\int_t^{t+\Delta t}\int_{-\infty}^{\infty} \int_{|y|\ge 1} \left[ \phi\left( H^{-1} (H(x)+ y) \right)-\phi(x)\right]p(x,s)\,{\nu}({\rm d} y)\, {\rm d} x\, {\rm d}s\nonumber\\
&\quad + \int_t^{t+\Delta t} \int_{-\infty}^{\infty} \int_{|y|< 1} \left[ \phi\left(H^{-1} (H(x)+ y)\right)-\phi(x)-   \phi'(x)\sigma(x) \,y \right]p(x,s)\,\nu ( {\rm d} y)\,{\rm d}x\,{\rm d}s\,.\nonumber\\
\end{align}
\normalsize
To derive the above equation, we have used the following facts \cite{Applebaum2009}:
\begin{align}\label{s1_16}
{\mathbf E} \left\{ \int_t^{t+\Delta t}  \phi' (X(s-) \sigma (X(s))  \, {\rm d} B(s)\right\} =0,
\end{align}
\begin{align}\label{s1_18a}
 {\mathbf E} \left\{\int_t^{t+\Delta t}\int_{|y|\ge 1} \left[ \phi\left( H^{-1} (H(X(s-))+ y) \right)-\phi(X(s-))\right]\,\tilde N({\rm d}s, {\rm d} y)\right\}=0,
\end{align}
\begin{align}\label{s1_17}
{\mathbf E} \left\{\int_t^{t+\Delta t}\int_{|y|< 1} \left[ \phi\left(H^{-1} (H(X(s-))+ y)\right)-\phi(X(s-))\right]\,\widetilde N({\rm d}s, {\rm d} y)\right\} = 0,
\end{align}
and
\begin{align}\label{s1_18}
&{\mathbf E} \left\{\int_t^{t+\Delta t}\int_{|y|\ge 1} \left[ \phi\left( H^{-1} (H(X(s-))+ y) \right)-\phi(X(s-))\right]\,N({\rm d}s, {\rm d} y)\right\}\nonumber\\
&\quad = {\mathbf E} \left\{\int_t^{t+\Delta t}\int_{|y|\ge 1} \left[ \phi\left( H^{-1} (H(X(s-))+ y) \right)-\phi(X(s-))\right]\,  \nu ({\rm d} y)\,{\rm d}s\right\}\nonumber\\
&\quad =   \int_t^{t+\Delta t} \int_{-\infty}^{\infty} \int_{|y|\ge 1} \left[ \phi\left( H^{-1} (H(x)+ y) \right)-\phi(x)\right]p(x,s)\, \nu ({\rm d} y)\,{\rm d}x\,{\rm d}s\,.\nonumber\\
\end{align}
Note that the first identity in (\ref{s1_18}) follows from (\ref{s1_18a}).

Equation (\ref{s1_15}) can be rewritten as
\small
\begin{align}\label{s1_19a}
& \int_t^{t+\Delta t} \int_{-\infty}^{\infty} \phi (x) \dfrac{{\rm d}p(x,s)}{{\rm d}s} \,{\rm d}x \,{\rm d}s\nonumber\\
 &= \int_t^{t+\Delta t} \int_{-\infty}^{\infty} \phi' (x) \left[f(x)p(x,s) + b\sigma (x) p(x,s) + \frac{A}{2}   \sigma  (x) \sigma'(x) p(x,s) \right]\,{\rm d}x\,{\rm d}s\nonumber\\
 &\quad
  + \frac{A}{2} \int_t^{t+\Delta t} \int_{-\infty}^{\infty} \phi'' (x)\sigma ^2(x)p(x,s)\,{\rm d}x \,{\rm d}s \nonumber\\
&\quad +\int_t^{t+\Delta t}\int_{-\infty}^{\infty} \int_{{\mathbf R}\backslash \{0\}} \left\{  \phi\left( H^{-1} (H(x)+ y) \right)-\phi(x)  -   \phi'(x)\,\sigma(x) \,y\,{\mathbf I}_{(-1,\;1)} (y) \right\} p(x,s)\,\nu ( {\rm d} y)\,{\rm d}x\,{\rm d}s,
\end{align}
 \normalsize
 where ${\mathbf I}_{(-1,\;1)} (y)$ is the indicator function.

 Since (\ref{s1_19a}) is true for any $t$ and $\Delta t$,  it follows that
 \small
 \begin{align}\label{s1_19}
& \int_{-\infty}^{\infty} \phi (x) \frac{\partial p(x,t)}{\partial t} {\rm d}x \nonumber\\
 &=   \int_{-\infty}^{\infty} \phi' (x) \left[f(x)p(x,t) + b\sigma (x) p(x,t) + \frac{A}{2}   \sigma  (x) \sigma'(x) p(x,t) \right]\, {\rm d}x\nonumber\\
 &\quad
  + \frac{A}{2}  \int_{-\infty}^{\infty} \phi'' (x)\sigma ^2(x)p(x,t) \, {\rm d}x \nonumber\\
&\quad + \int_{-\infty}^{\infty} \,{\rm d}x\, \int_{{\mathbf R}\backslash \{0\}} \left\{  \phi\left( H^{-1} (H(x)+ y) \right)-\phi(x)  -   \phi'(x)\sigma(x) \,y\,{\mathbf I}_{(-1,\;1)} (y) \right\}p(x,t)\,\nu ( {\rm d} y) .
\end{align}
 \normalsize
 By integration by parts, the first two integrals in the right hand side of (\ref{s1_19}) become
 \begin{align}\label{s1_20}
 &\int_{-\infty}^{\infty} \phi' (x) \left(f(x)  + b\sigma (x)  + \frac{A}{2}   \sigma  (x) \sigma'(x) \right)\, p(x,t) {\rm d}x \nonumber\\
 &=- \int_{-\infty}^{\infty} \phi (x) \frac{\partial}{\partial x} \left[\left(f(x)  + b\sigma (x)  + \frac{A}{2}   \sigma  (x) \sigma'(x)\right) p(x,t) \right]\, {\rm d}x,
\end{align}
and
\begin{align}\label{s1_21}
  \int_{-\infty}^{\infty} \phi'' (x)\sigma ^2(x)p(x,t) \, {\rm d}x
    =   \int_{-\infty}^{\infty} \phi (x) \frac{\partial^2}{\partial x^2}\left[\sigma ^2(x)p(x,t)\right] \, {\rm d}x \,,
\end{align}
respectively.

By  interchanging order of integrals, the last term of (\ref{s1_19}) becomes
\begin{align}\label{s1_22}
&\int_{-\infty}^{\infty}{\rm d}x \int_{{\mathbf R}\backslash \{0\}} \left[  \phi\left( H^{-1} (H(x)+ y) \right)-\phi(x)  -   \phi'(x)\sigma(x) \,y\, {\mathbf I}_{(-1,\;1)} (y) \right]p(x,t)\,\nu ( {\rm d} y)\nonumber\\
&= \int_{{\mathbf R}\backslash \{0\}} \nu ( {\rm d} y) \int_{-\infty}^{\infty} \left[  \phi\left( H^{-1} (H(x)+ y) \right)-\phi(x)  -   \phi'(x)\,\sigma(x) \,y\, {\mathbf I}_{(-1,\;1)} (y) \right]p(x,t)\,{\rm d}x
\end{align}
The interchanging order of integrals above is justified by
\begin{align}\label{s1_23}
&\int_{-\infty}^{\infty} \int_{{\mathbf R}\backslash \{0\}} \bigg|\left[  \phi\left( H^{-1} (H(x)+ y) \right)-\phi(x)  -   \phi'(x)\,\sigma(x) \,y\,{\mathbf I}_{(-1,\;1)} (y) \right]p(x,t)\bigg|\,\nu ( {\rm d} y)\,{\rm d}x\nonumber\\
& \le \int_{-\infty}^{\infty} \int_{|y|<1} \bigg|\left[  \phi\left( H^{-1} (H(x)+ y) \right)-\phi(x)  -   \phi'(x)\,\sigma(x)\,y\,   \right]p(x,t)\bigg|\,\nu ( {\rm d} y)\,{\rm d}x\nonumber\\
&\quad\quad +\int_{-\infty}^{\infty} \int_{{|y|\ge 1}} \bigg|\left[  \phi\left( H^{-1} (H(x)+ y) \right)-\phi(x)  \right]p(x,t)\bigg|\,\nu ( {\rm d} y)\,{\rm d}x\nonumber\\
&< +\infty.
\end{align}
To prove the last inequality in (\ref{s1_23}), we have used (\ref{s1_1}) and the fact that $\phi(x)~\in~C_0^\infty~({\mathbb R})$.

Next, let us examine the integral inside  the last term of (\ref{s1_22}), which can be written as
\begin{align}\label{s1_24}
& \int_{-\infty}^{\infty} \left[  \phi\left( H^{-1} (H(x)+ y) \right)-\phi(x)  -   \phi'(x) \sigma(x) \,y\,{\mathbf I}_{(-1,\;1)} (y) \right]p(x,t)\,{\rm d}x\nonumber\\
&\quad = \int_{-\infty}^{\infty}    \phi\left( H^{-1} (H(x)+ y) \right) p(x,t)\,{\rm d}x - \int_{-\infty}^{\infty}    \phi (x)  p(x,t)\,{\rm d}x \nonumber\\
& \quad \quad - \int_{-\infty}^{\infty}   \phi'(x)\, \sigma(x)\,y\, {\mathbf I}_{(-1,\;1)} (y)  p(x,t)\,{\rm d}x.
 \end{align}
 Denote
 \begin{align}\label{s1_25}
 z=  H^{-1} (H(x)+ y),
 \end{align}
 It follows from (\ref{s1_25}) and (\ref{s1_10}) that
 \begin{align}\label{s1_26}
 x=H^{-1} (H(z)- y),\quad \frac{{\rm d} x}{{\rm d} z} = \frac{\sigma(H^{-1} (H(z)- y)}{\sigma (z)}.
 \end{align}
 For the first integral in the right hand side of (\ref{s1_24}), by the change of variable and using (\ref{s1_26}) and (iii) in Lemma 1,  we can get
 \begin{align}\label{s1_28}
 & \int_{-\infty}^{\infty}    \phi\left( H^{-1} (H(x)+ y) \right) p(x,t)\,{\rm d}x \nonumber\\
 &= \int_{-\infty}^{\infty}    \phi ( z )\frac{\sigma(H^{-1} (H(z)- y) ) }{\sigma(z)} p(H^{-1} (H(z)- y),t)\,{\rm d}z \nonumber\\
 &= \int_{-\infty}^{\infty}    \phi (x)   \frac{\sigma(H^{-1} (H(x)- y)) }{\sigma(x)} p(H^{-1} (H(x)- y),t)\,{\rm d}x.
 \end{align}
  For the last integral at the right hand side of (\ref{s1_24}), we have
 \begin{align}\label{s1_29}
 \int_{-\infty}^{\infty}   \phi'(x) \,\sigma(x)\,y {\mathbf I}_{(-1,\;1)} (y)  p(x,t)\,{\rm d}x = -\int_{-\infty}^{\infty}   \phi(x)\, y \,{\mathbf I}_{(-1,\;1)} (y)  \frac{\partial}{\partial x} (\sigma(x)\,p(x,t))\,{\rm d}x.
 \end{align}
 Substituting (\ref{s1_28}) and (\ref{s1_29}) into (\ref{s1_24}), we obtain
\begin{align}\label{s1_30}
& \int_{-\infty}^{\infty} \left[  \phi\left( H^{-1} (H(x)+ y) \right)-\phi(x)  -   \phi'(x) \,\sigma(x)\,y\,{\mathbf I}_{(-1,\;1)} (y) \right]p(x,t)\,{\rm d}x\nonumber\\
 &  = \int_{-\infty}^{\infty}    \phi (x)  \left[ \frac{\sigma(H^{-1} (H(x)- y)) }{\sigma(x)} p(H^{-1} (H(x)- y),t) \right.\nonumber\\
 &\quad\quad\quad\quad\quad\quad \left. - p(x,t) +  y \,{\mathbf I}_{(-1,\;1)} (y)  \frac{\partial}{\partial x}\left(\sigma(x) p(x,t)\right) \right]{\rm d}x.
 \end{align}
 Substituting (\ref{s1_30}) into (\ref{s1_22}), we get
 \begin{align}\label{s1_31}
&\int_{-\infty}^{\infty}{\rm d}x \int_{{\mathbf R}\backslash \{0\}} \left[  \phi\left( H^{-1} (H(x)+ y) \right)-\phi(x)  -   \phi'(x) \sigma(x)\,y\, {\mathbf I}_{(-1,\;1)} (y) \right]p(x,t)\,\nu ( {\rm d} y)\nonumber\\
&= \int_{{\mathbf R}\backslash \{0\}} \nu ( {\rm d} y) \int_{-\infty}^{\infty}    \phi (x)  \left[ \frac{\sigma(H^{-1} (H(x)- y)) }{\sigma(x)} p(H^{-1} (H(x)- y),t) \right.\nonumber\\
 &\quad\quad\quad\quad\quad\quad \left. - p(x,t) +  y \,{\mathbf I}_{(-1,\;1)} (y)  \frac{\partial}{\partial x}(\sigma(x) p(x,t)) \right]{\rm d}x.
\end{align}
By interchanging order of integrals, (\ref{s1_31}) becomes
 \begin{align}\label{s1_32}
&\int_{-\infty}^{\infty}{\rm d}x \int_{{\mathbf R}\backslash \{0\}} \left[  \phi\left( H^{-1} (H(x)+ y) \right)-\phi(x)  -   \phi'(x) \,\sigma(x)\,y\, {\mathbf I}_{(-1,\;1)} (y) \right]p(x,t)\,\nu ( {\rm d} y)\nonumber\\
&= \int_{-\infty}^{\infty}{\rm d}x \int_{{\mathbf R}\backslash \{0\}}    \phi (x)  \left[ \frac{\sigma(H^{-1} (H(x)- y)) }{\sigma(x)} p(H^{-1} (H(x)- y),t) \right.\nonumber\\
 &\quad\quad\quad\quad\quad\quad \left. - p(x,t) +  y \,{\mathbf I}_{(-1,\;1)} (y)  \frac{\partial}{\partial x}(\sigma(x) p(x,t)) \right]\nu({\rm d}y).
\end{align}
Substituting (\ref{s1_20}), (\ref{s1_21}), and (\ref{s1_32}) into (\ref{s1_19}), we get
  \small
 \begin{align}\label{s1_33}
& \int_{-\infty}^{\infty} \phi (x) \frac{\partial p(x,t)}{\partial t} {\rm d}x \nonumber\\
   &=- \int_{-\infty}^{\infty} \phi (x) \frac{\partial}{\partial x} \left[\left(f(x)  + b\sigma (x)  + \frac{A}{2}   \sigma  (x) \sigma'(x)\right) p(x,t) \right]\, {\rm d}x\nonumber\\
 &\quad
  + \frac{A}{2}  \int_{-\infty}^{\infty} \phi (x) \frac{\partial^2}{\partial x^2}\left[\sigma ^2(x)p(x,t)\right] \, {\rm d}x \nonumber\\
&\quad + \int_{-\infty}^{\infty}{\rm d}x \int_{{\mathbf R}\backslash \{0\}}    \phi (x)  \left[ \frac{\sigma(H^{-1} (H(x)- y)) }{\sigma(x)} p(H^{-1} (H(x)- y),t) \right.\nonumber\\
 &\quad\quad\quad\quad\quad\quad\quad\quad\quad \left. - p(x,t) +  y \,{\mathbf I}_{(-1,\;1)} (y)  \frac{\partial}{\partial x}(\sigma(x) p(x,t)) \right]\nu({\rm d}y)
\end{align}
\normalsize
Since the above equation is true for any $\phi(x) \in C_0^\infty ({\mathbb R})$, it follows from Assumption H1 and Lemma \ref{lemma2} that the probability density function $p$ satisfies (\ref{s1_34}).
\end{proof}

\subsection{Cases where $\sigma$ has zeros}
When $\sigma$ has zeros, the transform $H$ given in (\ref{s1_10}) is not well defined.  Suppose $\sigma$ has $n$ zeros represented by $x_i$ ($i=1, 2, \cdots, n$). We denote  $x_0 = -\infty$ and $x_{n+1}=+\infty$ for convenience. Without loss of generality, we suppose
\begin{align}\label{s1_35}
-\infty = x_0 <x_1 < \cdots < x_n <x_{n+1} = +\infty.
\end{align}
\begin{definition}
The transforms $H_i$($i=0,1,2,\cdots,n$) are defined as
\begin{align}\label{s1_36}
H_i (x) =  \int_{a_i} ^x \frac{{\rm d}y}{\sigma(y)}\quad\quad \text{for\quad $x\in (x_i, x_{i+1})$},
\end{align}
where $a_i$ is  any  constant in $(x_i, x_{i+1})$.
\end{definition}
The transforms $\{H_i\}$ defined in (\ref{s1_36})  have the following properties.
\begin{lemma}
Suppose $\sigma$ is Lipschitz and has n zeros, then each transform $H_i$($i=1,2,\cdots,n$) defined by (\ref{s1_36}) has the following properties:\\
(i) each $H_i$ is well defined and monotone on $(x_{i},x_{i+1})$;\\
(ii) each $H_i$ is bijective and maps from $(x_{i},x_{i+1})$ to $(-\infty,\infty)$;\\
(iii) each $H_i$ has the inverse  $H_i^{-1}$, and  $H_i^{-1}$ is  bijective and maps from $(-\infty,\infty)$ to $(x_{i},x_{i+1})$;\\
(iv) $\forall y\in \mathbb{R}$, each $H_i^{-1}(H_i(\cdot)+y)$  is bijective from $(x_{i},x_{i+1})$ to $(x_{i},x_{i+1})$, and has the following property
\begin{align}\label{s1_37}
\begin{cases}
\lim_{x \to x_i +} H_i^{-1} (H_i (x) +y) = x_i,\\
\lim_{x \to x_{i+1}-} H_i^{-1} (H_i (x)+y) = x_{i+1},
\end{cases}
\end{align}
where $\lim_{x \to x_i +}$ represents the right limit at $x=x_i$, and $\lim_{x \to x_{i+1}-}$ the left limit at $x=x_{i+1}$.
\end{lemma}
\begin{proof}[Proof of Lemma 2]
(i) follows from the fact that $\sigma(x) \ne 0, \; \forall x\in (x_i, x_{i+1})$.

To show (ii), since $H_i$ is monotone and defined on $(x_i, x_{i+1})$, it is sufficient to show that $H_i$ goes to infinity as $x$ approaches $x_i$ or $x_{i-1}$. Without loss of generality, we suppose $\sigma(x)>0$ for $x\in (x_i,x_{i+1})$.

For $H_0$ and $H_n$, we have
\begin{align}\label{a2_1}
H_0(x) &= \int_{a_0}^x \frac{{\rm d} t}{|\sigma(t)|} = \int_{a_0}^x \frac{{\rm d} t}{ |\sigma(t)-\sigma(x_1) |}  \ge \int_{a_0}^x \frac{{\rm d} t}{L| t - x_1|  },\quad\quad \forall x\in (x_0,x_1),
\end{align}
and
\begin{align}\label{a2_1a}
H_n(x) &= \int_{a_n}^x \frac{{\rm d} t}{|\sigma(t)|} = \int_{a_n}^x \frac{{\rm d} t}{ |\sigma(t)-\sigma(x_n) |}  \ge \int_{a_n}^x \frac{{\rm d} t}{L| t - x_n|  },\quad\quad \forall x\in (x_n,x_{n+1}),
\end{align}
respectively.

For $H_i$ ($i=1,2,\cdots,n-1$), we have
\begin{align}\label{a2_1b}
H_i(x) &= \int_{a_i}^x \frac{{\rm d} t}{|\sigma(t)|} = \int_{a_i}^x \frac{{\rm d} t}{ |\sigma(t)-\sigma(x_i)| }  \ge \int_{a_i}^x \frac{{\rm d} t}{L| t - x_i|  }, \quad\quad \forall  x\in (x_i, x_{i+1}),
\end{align}
and
\begin{align}\label{a2_1c}
H_i(x)&= \int_{a_i}^x \frac{{\rm d} t}{|\sigma(t)|} = \int_{a_i}^x \frac{{\rm d} t}{ |\sigma(t)-\sigma(x_{i+1})| }  \ge \int_{a_i}^x \frac{{\rm d} t}{L| t - x_{i+1}|  }, \quad\quad \forall  x\in (x_i, x_{i+1}),
\end{align}
It follows from (\ref{a2_1}) to (\ref{a2_1c}) that $H_i(x)$($i=0,1,2,\cdots,n$) goes to infinity as $x$ approaches $x_i$ or $x_{i-1}$. Now, the proof of (ii) is finished.

 (iii) and (iv) are implied by (ii).
\end{proof}

\begin{definition}
The transform $\tilde H: \mathbb{R}\times\mathbb{R}\to \mathbb{R}$ associated with SDE (\ref{s1_5}) is defined by
\begin{align}\label{s1_39}
\tilde H(x,y)=
\begin{cases}
H_i^{-1} \big(H_i (x) + y\big), &\text{for $x \in (x_i,x_{i+1})$ and $y \in \mathbb{R}$,}\\
x, &\text{for $x= x_1, x_2, \cdots, x_n $ and $y \in \mathbb{R}$},
\end{cases}
\end{align}
\end{definition}
The transform $\tilde H$ has the following properties.
\begin{lemma}\label{lemma4}
Suppose $\sigma$ is a real analytic function (i.e., it possesses derivatives of all orders and its function value agrees with its Taylor series in a neighborhood of every point) and has $n$ zeros denoted as in (\ref{s1_35}), then for any given $y\in \mathbb{R}$, $\tilde H(x,y)$ is continuously differentiable with respect to $x$ for all $x \in \mathbb{R}$. Moreover,
\begin{align}\label{s1_41c}
 \frac{{\rm \partial} \tilde H(x, y)}{{\rm \partial}x}=\begin{cases}\frac{\sigma(\tilde H(x, y))}{\sigma (x)}  &\quad\quad \text{for $x\in \bigcup\limits_{i=0}^{n} (x_i,x_{i+1})$},\\ \sum \limits_{k=0}^{\infty} \dfrac{1}{k!} \Phi_k(x) y^k &\quad\quad \text{for $x=x_1,x_2,\cdots,x_n$,}\end{cases}
\end{align}
where $\Phi_k(x_i)$ is defined as
\begin{align}\label{a1_5}
\begin{cases}
\Phi_0(x_i) = 1,\\
\Phi_k(x_i) = \lim_{x\to x_i} \underbrace{\frac{{\rm d}}{{\rm d}x}\left(\sigma (x) \frac{{\rm d}}{{\rm d}x}\left(\sigma (x) \cdots \left(\frac{{\rm d}}{{\rm d}x}\sigma (x)\right)\right)\right)}_{k-fold}\quad\quad (k=1, 2, \cdots).
\end{cases}
\end{align}
\end{lemma}
\begin{proof}[Proof of Lemma 3]

For $x \ne x_i(i=1, 2, \cdots, n)$, since $\sigma$ is analytic, it is straightforward to check that $\tilde H(x,y)$ is continuously differentiable with respect to $x$, and by direct computation we can get
\begin{align}\label{a1_3b}
\frac{\partial \tilde H(x, y)}{\partial x}= \frac{\sigma(\tilde H(x, y))}{\sigma (x)}  &\quad\quad \text{for $x\in \bigcup\limits_{i=0}^{n} (x_i,x_{i+1})$}.
\end{align}.

In the following, we show that $\tilde H(x,y)$ is continuously differentiable with respect to $x$ for $x=x_1,x_2,\cdots,x_n$.

First, we see from (\ref{s1_37}) that $\tilde H(x,y)$ is continuous at  $x=x_1,x_2,\cdots,x_n$. Since $\sigma$ is analytic and $\frac{\partial \tilde H(x,y)}{\partial y}=\sigma(\tilde H(x,y))$, by Taylor expansion with respect to $y$ at $y=0$ and $x\ne x_i$ ($i=1, 2, \cdots, n$), we have
\begin{align}\label{a1_3}
\tilde H(x,y) =  x + \sigma(x) y + \frac{1}{2!}  \sigma(x) \left(\frac{{\rm d}}{{\rm d}x} \sigma(x)\right)   y^2 + \frac{1}{3!}  \sigma(x) \left(\frac{{\rm d}}{{\rm d}x} \left(\sigma(x) \frac{{\rm d}}{{\rm d}x} \sigma(x)\right)\right)  y^3 + \cdots\;
\end{align}
and
\begin{align}\label{a1_3a}
\sigma(\tilde H(x,y)) =  \sigma(x) +  \sigma(x) \left(\frac{{\rm d}}{{\rm d}x} \sigma(x)\right)   y  + \frac{1}{2!}  \sigma(x) \left(\frac{{\rm d}}{{\rm d}x} \left(\sigma(x) \frac{{\rm d}}{{\rm d}x} \sigma(x)\right)\right)  y^2 + \cdots\;.
\end{align}
By using (\ref{a1_3}), we get
\begin{align}\label{a1_4}
\frac{\partial  \tilde H(x,y)}{\partial x}\bigg |_{x=x_i}& =  \lim_{x\to x_i} \frac{\tilde H(x,y) - \tilde H(x_i,y)}{x-x_i}\nonumber\\ &=\lim_{x\to x_i} \frac{\tilde H(x,y) -x_i}{x-x_i} = \sum _{k=0}^{\infty} \frac{1}{k!} \Phi_k(x_i) y^k,
\end{align}
where $\Phi_k(x_i)$ is defined in (\ref{a1_5}), and the convergence of the infinite series can be checked straightforwardly by using the fact that $\sigma$ is Lipshitz continuous.  (\ref{a1_4}) indicates that $\tilde H(x,y)$ is differentiable at $x = x_i$ ($i=1, 2, \cdots, n$).

It follows from (\ref{a1_3a}) that
\begin{align}\label{s1_41b}
\lim _{x \to x_i} \frac{\sigma(\tilde H(x, y))}{\sigma (x)}  = \sum _{k=0}^{\infty} \frac{1}{k!} \Phi_k(x_i) y^k.
\end{align}
It follows from (\ref{a1_3b}),  (\ref{a1_4}) and (\ref{s1_41b}) that $\tilde H(x,y)$ is continuously differentiable with respect to $x$.
\end{proof}

\begin{proposition}\label{proposition2}
Suppose the assumption H1 holds,  $f$ is differentiable, $\sigma$ is Lipschitz continuous, analytic, and has $n$ zeros $\{x_i\}_{i=1,2,\cdots,n}$ as defined in (\ref{s1_35}), then the probability density function $p(x,t)$ for the solution $X(t)$ of the SDE (\ref{s1_5}) satisfies the following equation
 \begin{align}\label{s1_42}
  \frac{\partial p(x,t)}{\partial t}& =-  \frac{\partial}{\partial x} \left[\left(f(x)  + b\sigma (x)   + \frac{A}{2}   \sigma  (x) \sigma'(x)\right) p(x,t) \right] + \frac{A}{2}   \frac{\partial^2}{\partial x^2}\left[\sigma ^2(x)p(x,t)\right]\nonumber\\
    &+ \int_{{\mathbf R}\backslash \{0\}}     \left[ \frac{\partial \tilde H(x,-y)}{\partial x} p(\tilde H(x, -y),t) \right.\nonumber\\
 &\quad\quad\quad\quad\quad\quad\quad\quad\quad \left. - p(x,t) +   y \,{\mathbf I}_{(-1,\;1)} (y)  \frac{\partial}{\partial x}(\sigma(x) p(x,t)) \right]\nu({\rm d} y),
\end{align}
where $\tilde H(x,y)$ is defined in (\ref{s1_39}) and (\ref{s1_36}), and  $\frac{\partial \tilde H(x,y)}{\partial x}$ is given in (\ref{s1_41c}).
\end{proposition}

\begin{proof}[Proof of Proposition \ref{proposition2}]

The proof follows the same steps as the proof of Proposition \ref{proposition1} presented in subsection 2.1. Here, we only state the difference.

When $\sigma$ has $n$ zeros as defined in (\ref{s1_35}), equation (\ref{s1_12}) in subsection 2.1 now changes to
\begin{align}\label{s1_38}
\begin{cases}
\xi(\Delta L(t),  X(t-)) = H_i^{-1} \big(H_i (X(t-)) + \Delta L(t)\big),  &\quad\quad\text{for $X(t-) \in (x_i, x_{i+1})$},\\
\xi(\Delta L(t),   X(t-))= X(t-).  &\quad\quad \text{for $X(t-) = x_1, x_2, \cdots, x_n$}.
\end{cases}
\end{align}
With the help of $\tilde H$ defined in (\ref{s1_39}),  (\ref{s1_38}) can be written as
\begin{align}\label{s1_40}
\xi(\Delta L(t), X(t-)) = \tilde H(X(t-), \Delta L(t)).
\end{align}

Comparing $H^{-1}(H(x)+y))$ used  in the case $\sigma$ being nonzero with $\tilde H(x, y)$, one can see that they are both continuously differential with respect to $x$, and the role of $H^{-1}(H(x)+y))$ in the proof of Proposition 1 can now be replaced completely by $\tilde H(x, y)$.  Replacing $H^{-1}(H(x)+y))$ in the proof of Proposition \ref{proposition1} with $\tilde H(x, y)$, we arrive at the equation (\ref{s1_42}).
 \end{proof}

\section{Examples}
In this section, we present a couple of simple examples to illustrate the results we have obtained in Proposition \ref{proposition2}.

\subsection{Example 1}

Let $\sigma(x)=x$, and $L(t)$ be the $\alpha$-stable L\'{e}vy motion with the triplet $b=1$, $A=0$ and $\nu({\rm d}x)=\frac{ {\rm d}x}{|x|^{1+\alpha}} $. For more details about $\alpha$-stable L\'{e}vy motion, see \cite{Sato1999} among others. Then the SDE (\ref{s1_5}) becomes
\begin{equation}\label{exmp1_1}
{\rm d}X(t)=f(X(t)){\rm d}t+X(t)\diamond {\rm d}L(t),
\end{equation}
According to (\ref{s1_36}),
\begin{align}\label{exmp1_2}
H_0(x)=\ln \frac{x}{a_0} &&\text{for $x\in (-\infty, 0)$},
\end{align}
and
\begin{align}\label{exmp1_3}
H_1(x)=\ln \frac{x}{a_1} &&\text{for $x\in (0, +\infty)$},
\end{align}
where $a_0$ and $a_1$ are any given constants satisfying $a_0 \in (-\infty,0)$ and $a_1 \in (0, +\infty)$.
It follows from (\ref{exmp1_2}) and (\ref{exmp1_3})  that
\begin{align}\label{exmp1_4}
H_0^{-1}(x)=a_0 e^x &&\text{for $x\in (-\infty, 0)$},
\end{align}
and
\begin{align}\label{exmp1_5}
H_1^{-1}(x)=a_1 e^x &&\text{for $x\in (0, +\infty)$}.
\end{align}
Therefore,
\begin{align}\label{exmp1_6}
H_0^{-1}(H_0(x)+y)=xe^y &&\text{for $x\in (-\infty, 0)$},
\end{align}
and
\begin{align}\label{exmp1_7}
H_1^{-1}(H_1(x)+y)=x e^y &&\text{for $x\in (0, +\infty)$}.
\end{align}
With (\ref{exmp1_6}) and (\ref{exmp1_7}), $\tilde H$ defined by (\ref{s1_39}) becomes
\begin{align}\label{exmp1_8}
\tilde H(x,y)=\begin{cases} xe^y &\quad\quad \text{for $x \ne 0$},\\0 &\quad\quad \text{for $x=0$}.\end{cases}
\end{align}
Namely,
\begin{align}\label{exmp1_9}
\tilde H(x,y)=  xe^y.
\end{align}
Therefore, according to (\ref{s1_42}) in Proposition \ref{proposition2},  Fokker-Planck equation for (\ref{exmp1_1}) can be expressed as
 \begin{align}\label{exmp1_10}
  \frac{\partial p(x,t)}{\partial t}& =-  \frac{\partial}{\partial x} \left[f(x)p(x,t) + xp(x,t)  \right]\nonumber\\& + \int_{{\mathbf R}\backslash \{0\}}     \left[ e^{-y} p(xe^{-y},t)
 - p(x,t) +   y \,{\mathbf I}_{(-1,\;1)} (y)  \frac{\partial}{\partial x}(x p(x,t)) \right]\frac{  {\rm d} y}{|y|^{1+\alpha}}.
\end{align}

\subsection{Example 2}
All the parameter in this example are the same as those in Example 1 in the previous subsection except that a combined Gaussian and white noise is used in (\ref{exmp1_1})  instead of $\alpha$-stable white noise. A combined Gaussian and Poisson white noise is corresponding to a L\'evy process which consists of two components: (i) Brownian motion; (ii) compound Poisson process, and can be expressed as
\begin{align}\label{exmp2_1}
L(t)=B(t) + \sum\limits_{i=0}^{N(t)}  r_i,
\end{align}
where $B(t)$ is a standard scalar Brownian motion with variance matrix $\tilde A$, $N(t)$ ($t>0$) is a Poisson process with intensity parameter $\lambda$, $r_i$ ($i=1,2,\cdots$) are i.i.d random numbers, with probability density function  $\mu(x)$, which are also independent of $N(t)$.  The L\'evy process expressed in (\ref{exmp2_1}) has a triplet as $b=-\lambda \int_{|y|<1} y\mu({{\rm d}y})$, $A=\tilde A$, and $\nu({\rm d} x) = \lambda \mu({\rm d} y)$ \cite{Applebaum2009}.

Same as Example 1,  we get the following Fokker-Planck equation for SDE (\ref{exmp1_1})
 \begin{align}\label{exmp2_2}
  \frac{\partial p(x,t)}{\partial t}& =-  \frac{\partial}{\partial x} \left[\left(f(x)+x+\frac{\tilde A}{2} x\right) p(x,t)  \right]+\frac{\tilde A}{2} \left(\sigma^2(x) p(x,t)\right)\nonumber\\& +\lambda \int_{{\mathbf R}\backslash \{0\}}     \left[ e^{-y} p(xe^{-y},t)
 - p(x,t) +   y \,{\mathbf I}_{(-1,\;1)} (y)  \frac{\partial}{\partial x}(x p(x,t)) \right]\mu({\rm d}y).
\end{align}

\section{Conclusion}
We derived Fokker-Planck equations for one-dimensional stochastic systems modeled by Marcus SDEs driven by L\'{e}vy processes.  The main results are summarized as in  Proposition 1  in subsection 2.1 and Proposition 2 in subsection 2.2.  
The derived Fokker-Planck equations are essentially non-local partial differential equations and may involve some singular integrals depending on the specific L\'{e}vy processes. It is a  challenging but important task to develop some efficient numerical methods for these Fokker-Planck equations, which will be left for our future research.

\section*{Acknowledgment}
This work is supported by  National Natural Science Foundation of China (No. 11531006).

\bibliographystyle{plain}

\begin{thebibliography}{10}

\bibitem{Applebaum2009}
D.~Applebaum.
\newblock {\em L\'evy Processes and stochastic calculus}.
\newblock Cambridge University Press, 2nd Edition, 2009.

\bibitem{Bhattacharyya2012}
P.~Bhattacharyya.
\newblock {\em Distributions: Generalized functions with applications in
  Sobolev spaces}.
\newblock De Gruyter, 2012.

\bibitem{ChechkinPavlyukevich2014}
A.~Chechkin and I.~Pavlyukevich.
\newblock Marcus versus Stratonovich for systems with jump noise.
\newblock {\em J. Phys. A: Math. Theor.}, 47:342001, 2014.

\bibitem{DiPaolaFalsone1993}
M.~Di~Paola and G.~Falsone.
\newblock Ito and Stratonovich integrals for Delta-correlated processes.
\newblock {\em Probabilistic Engineering Mechanics}, 8:197--208, 1993.

\bibitem{DiPaolaFalsone1993b}
M.~Di~Paola and G.~Falsone.
\newblock Stochastic dynamics of non-linear systems driven by non-normal
  Delta-correlated processes.
\newblock {\em ASME Journal of Applied Mechanics}, 60:141--148, 1993.


\bibitem{KanazawaSagawaHayakawa2012}
K.~Kanazawa, T.~Sagawa, and H.~Hayakawa.
\newblock Stochastic energetics for non-Gaussian processes.
\newblock {\em Physical Review Letters}, 108:201601, 2012.

\bibitem{Klebaner2005}
F.~C. Klebaner.
\newblock {\em Introduction to stochastic calculus with applications}.
\newblock 2nd Edition, Imperial College Press, 2005.

\bibitem{KurtzPardouxProtter1995}
T.~Kurtz, E.~Pardoux, and P.~Protter.
\newblock Stratonovich stochastic differential equations driven by general
  martingales.
\newblock {\em Ann. Inst. Henri Poincare Prob. Stat.}, 31:351--377, 1995.

\bibitem{Marcus1978}
S.~I. Marcus.
\newblock Modeling and analysis of stochastic differential equations driven by
  point processes.
\newblock {\em IEEE Transanctions on Information Theory}, IT-24:164--172, 1978.

\bibitem{Marcus1981}
S.~I. Marcus.
\newblock Modeling and approximation of stochastic differential equations
  driven by semimartingales.
\newblock {\em Stochastics}, 4:223--245, 1981.

\bibitem{Oksendal2003}
B.~K. Oksendal.
\newblock {\em Stochastic Differential Equations : An Introduction with
  Applications}.
\newblock Springer, 6th Edition, 2003.

\bibitem{Sato1999}
K.~Sato.
\newblock {\em L\'evy Processes and infinitely divisible distributions}.
\newblock Cambridge University Press, 1999.

\bibitem{Schertzer2001}
D.~Schertzer, M.~Larcheveque, J.~Duan, V.~Yanovsky, and S.~Lovejoy.
\newblock Fractional Fokker-Planck equation for nonlinear stochastic
  differential equations driven by non-Gaussian L\'evy stable noises.
\newblock {\em Journal of Mathematical Physics}, 42:200--212, 2001.

\bibitem{SunDuan2012}
X.~Sun and J.~Duan.
\newblock Fokker-Planck equations for nonlinear dynamical systems driven by
  non-Gaussian L\'{e}vy processes.
\newblock {\em Journal of Mathematical Physics}, 53:072701, 2012.

\bibitem{SunDuanLi2013}
X.~Sun, J.~Duan, and X.~Li.
\newblock An alternative expression for stochastic dynamical systems with
  parametric Poisson white noise.
\newblock {\em Probabilistic Engineering Mechanics}, 32:1--4, 2013.

\bibitem{SunDuanLi2016}
X.~Sun, J.~Duan, and X.~Li.
\newblock Modeling nonlinear oscillators under excitation of combined Gaussian and Poisson white noise: a viewpoint based on engergy conservation law.
\newblock {\em Nonlinear Dynamics}, Accepted, 2016(DOI: 10.1007/s11071-015-2570-7).

\end{thebibliography}

\end{document}